\documentclass[11pt]{amsart}

\usepackage[utf8x]{inputenc}
\usepackage[english]{babel}
\usepackage[all,cmtip]{xy}

\usepackage[pdftex]{graphicx}
\usepackage[margin=2.8cm]{geometry}
\usepackage{amsfonts}
\usepackage{amsmath}
\usepackage{amssymb}
\usepackage{amsthm}
\usepackage{dsfont}
\usepackage{tikz}
\usepackage{pgfplots}
\usepackage{MnSymbol}
\usepackage{tikz}

\usepackage[T1]{fontenc}

\usepackage{paralist}

\usepackage[colorlinks,link color=blue, cite color=blue,pagebackref,pdftex]{hyperref}

\renewcommand\phi{\varphi}

\newcommand\R{\mathbb{R}}

\newcommand\hri[2]{^{\langle #1,#2\rangle}}

\newcommand\Urd[2]{\mathrm{U}_ {#1}^{#2}}

\DeclareMathOperator*{\Ehr}{E}

\newtheorem{thm}{Theorem}[section]
\newtheorem*{thm*}{Theorem}
\newtheorem{cor}[thm]{Corollary}
\newtheorem{lem}[thm]{Lemma}
\newtheorem{prop}[thm]{Proposition}
\newtheorem{obs}[thm]{Observation}

\newtheorem{quest}{Question}

\theoremstyle{definition}

\newtheorem{rem}[thm]{Remark}

\title[Symmetric decompositions and the Veronese construction]{Symmetric decompositions and the Veronese construction}

\author{Katharina Jochemko}
\address{Department of Mathematics, %
Royal Institute of Technology (KTH), %
Sweden}
\email{jochemko@kth.se}


\date{\today}

\begin{document}

\maketitle

\begin{abstract}
We study rational generating functions of sequences $\{a_n\}_{n\geq 0}$ that agree with a polynomial and investigate symmetric decompositions of the numerator polynomial for subsequences $\{a_{rn}\}_{n\geq 0}$. We prove that if the numerator polynomial for $\{a_n\}_{n\geq 0}$ is of degree $s$ and its coefficients satisfy a set of natural linear inequalities then the symmetric decomposition of the numerator for $\{a_{rn}\}_{n\geq 0}$ is real-rooted whenever $r\geq \max \{s,d+1-s\}$. Moreover, if the numerator polynomial for $\{a_n\}_{n\geq 0}$ is symmetric then we show that the symmetric decomposition for $\{a_{rn}\}_{n\geq 0}$ is interlacing. 

We apply our results to Ehrhart series of lattice polytopes. In particular, we obtain that the $h^\ast$-polynomial of every dilation of a $d$-dimensional lattice polytope of degree $s$ has a real-rooted symmetric decomposition whenever the dilation factor $r$ satisfies $r\geq \max \{s,d+1-s\}$. Moreover, if the polytope is Gorenstein then this decomposition is interlacing.
\end{abstract}

\section{Introduction}
For integers $r\geq 1$ we consider the operator $\Urd{r}{d}\colon \mathbb{R}[t]\rightarrow \mathbb{R}[t]$ defined in such a way that
\[
\sum _{n\geq 0} a_{rn}t^n =\frac{\Urd{r}{d} h(t)}{(1-t)^{d}}\quad \text{ whenever } \quad \sum _{n\geq 0} a_{n}t^n =\frac{h (t)}{(1-t)^{d}}
\]
for polynomials $h(t)\in \mathbb{R}[t]$. We investigate properties of $\Urd{r}{d} h (t)$ as a function of $r$. We study symmetric decompositions of $\Urd{r}{d+1} h(t)$ and their roots for polynomials $h(t)=h_0+h_1t+\cdots h_st^s$, $s\leq d$,  with nonnegative coefficients that satisfy both the following sets of inequalities
\begin{itemize}
\item[\hypertarget{H}{(H)}] $h_0+h_1+\cdots + h_{i}\geq h_d+h_{d-1}+\cdots +h_{d-i+1}$  for all $i$, and
\item[\hypertarget{S}{(S)}] $h_0+h_1+\cdots + h_{i}\leq h_s+h_{s-1}+\cdots +h_{s-i}$  for all $i$,
\end{itemize}
where $h_i:=0$ for all $i\not \in \{0,1,\ldots, s\}$.

Our motivation comes from Ehrhart theory, and more generally, commutative algebra. If $\sum _{n\geq 0}a_nt^n$ is the Hilbert series of a graded ring $R$ then $\sum _{n\geq 0}a_{nr}t^n$ is the Hilbert series of its $r$-th Veronese subalgebra $R^{\langle r\rangle}$. In Ehrhart theory, which is concerned with the enumeration of lattice points in dilates of lattice polytopes, $\Urd{r}{d+1} h^\ast _P(t)$ is equal the $h^\ast$-polynomial of the $r$-th dilate of a $d$-dimensional lattice polytope $P$. The Properties \hyperlink{H}{(H)} and \hyperlink{S}{(S)} naturally appear in this context: Property \hyperlink{S}{(S)} and the nonnegativity of the coefficients of $h(t)$  were proved by Stanley~\cite{StanleyS} for Hilbert series of semistandard graded Cohen-Macaulay domains. Property \hyperlink{H}{(H)} was proved to hold by Hibi~\cite{HibiH} for $h^\ast$-polynomials of lattice polytopes, using methods from commutative algebra.

The limiting behavior of $\Urd{r}{d} h(t)$ was studied by Beck and Stapledon~\cite{Beck2010}, Brenti and Welker~\cite{BrentiWelker}, and Diaconis and Fulman~\cite{Diaconis2} who proved under mild conditions on $h(t)$ that $\Urd{r}{d} h (t)$ eventually has only real roots. In~\cite{jochemko2018} it was shown that for every polynomial $h(t)$ with nonnegative coefficients $\Urd{r}{d} h(t)$ has only real roots whenever $r\geq \deg h$ thereby proving a uniform bound conjectured by Beck and Stapledon~\cite{Beck2010}. See also~\cite{Zhang} for related work. In the present article we provide a uniform bound on $r$ for which the symmetric decomposition of  $\Urd{r}{d} h(t)$ has only real zeros.

With basic linear algebra it can be observed that for every polynomial of degree at most $d$ there are uniquely determined symmetric polynomials $p(t)=t^dp(1/t)$ and $q(t)=t^{d-1}q(1/t)$ such that $h(t)=p(t)+tq(t)$. The pair $(p,q)$ is called the \textbf{symmetric decomposition} (with respect to degree $d$) of $h(t)$. In case $h(t)$ is the $h^\ast$-polynomial of a lattice polytope with an interior lattice point Betke and McMullen~\cite{BetkeMcMullen} gave a combinatorial interpretation for $p(t)$ and $q(t)$ from which the nonnegativity of their coefficients follows. Stapledon~\cite{stapledon2009inequalities} considered more general decompositions (see Lemma~\ref{lem:extendedStapledon}) in order to study inequalities amongst the coefficients of the $h^\ast$-polynomial of arbitrary lattice polytopes. In particular, he obtained purely combinatorial proofs for the  Properties \hyperlink{H}{(H)} and \hyperlink{S}{(S)} in case $h(t)$ is the $h^\ast$-polynomial of a lattice polytope.

Br\"and\'en and Solus~\cite{branden2018symmetric} recently initiated a systematic study of symmetric decompositions for which $p$ and $q$ have only nonnegative coefficients and only real roots. Such decompositions are called \textbf{nonnegative} and \textbf{real-rooted}. Real-rooted symmetric decompositions and real-rooted polynomials in general are of current particular interest in combinatorics  (see, e.g., \cite{athanasiadis2020,haglund2019real,hlavacek2020subdivisions,savage2015s}), and especially also in Ehrhart theory (see, e.g., \cite{jochemkozonotopes,ferroni2020ehrhart,higashitani19,SolusSimplices}), due to their applicability to unimodality questions.

A polynomial $h(t)=h_0+h_1t+\cdots + h_dt^d$ is called \textbf{unimodal} if $h_0\leq h_1 \leq \cdots \leq h_k \geq \cdots \geq h_d$ for some $k$. A central open conjecture in Ehrhart theory going back to Stanley~\cite{Stanleyunimodal} states that the $h^\ast$-polynomial of every lattice polytope having the integer decomposition property (IDP) is unimodal. Schepers and Van Langenhoven~\cite{Schepers} strenghtened this conjecture for IDP polytopes with interior lattice points by asking if every such polytope has an alternatingly increasing $h^\ast$-polynomial. The alternating increasing property is equivalent to both parts $p(t)$ and $q(t)$ in the symmetric decomposition being nonnegative and unimodal~\cite{jochemkozonotopes}. One way of proving that a polynomial with nonnegative coefficients is unimodal is by showing that it has only real roots. In fact, even stronger, if $h(t)=h_0+h_1t+\cdots + h_dt^d$ is real-rooted and has only nonnegative coefficients then it is \textbf{log-concave}, that is, $h_{i-1}h_{i+1}\leq h_i^2$ is satisfied for all $i$ (see, e.g., \cite[Lemma 1.1]{Brandenunimodal}). For further reading on real-rooted, log-concave and unimodal polynomials and their applications we recommend~\cite{Brandenunimodal,Braununimodal,Brentiunimodal,Stanleyunimodal}.

By a result of Bruns, Gubeladze and Trung~\cite{Bruns} every lattice polytope eventually becomes IDP under dilations. Hering~\cite{Hering} moreover showed that if $P$ is a lattice polytope whose $h^\ast$-polynomial has degree $s$ then $rP$ is IDP whenever $r\geq s$. On the other hand, $r=d+1-s$ is the smallest integer such $rP$ has an interior lattice point. In light of Schepers and Van Langenhoven's question~\cite{Schepers} it is natural to study the symmetric decomposition of the $h^\ast$-polynomial under dilation and, in particular, for dilation factors $r\geq \max\{s,d+1-s\}$. This corresponds to studying $\Urd{r}{d+1} h^\ast _P(t)$ for  $r\geq \max\{s,d+1-s\}$.

In Section~\ref{sec:symmetricdecomp} we derive an explicit formula for the symmetric decomposition of $\Urd{r}{d+1}h(t)$ (Proposition~\ref{prop:formula}) by generalizing the decompositions studied by Stapledon in \cite{stapledon2009inequalities} (Proposition~\ref{prop:decomp2}). In Section~\ref{sec:realrootedness}, we study the roots of these symmetric decompositions. The following is our main result.
\begin{thm}\label{thm:main}
Let $h(t)=h_0+h_1t+\cdots +h_st^s$ be a polynomial of degree $s\leq d$ with nonnegative coefficients such that
\begin{itemize}
\item[(H)] $h_0+h_1+\cdots + h_{i}\geq h_d+h_{d-1}+\cdots +h_{d-i+1}$  for all $i$, and
\item[(S)] $h_0+h_1+\cdots + h_{i}\leq h_s+h_{s-1}+\cdots +h_{s-i}$  for all $i$.
\end{itemize}
Then the polynomial $\Urd{r}{d+1} h(t)$ has a nonnegative real-rooted symmetric decomposition whenever $r\geq \max \{s, d+1-s\}$.
\end{thm}
This theorem strengthens results by Higashitani~\cite{higashitani2009} who proved that under the same conditions $\Urd{r}{d+1} h(t)$ has a log-concave symmetric decomposition whenever $r\geq \max \{s, d+1-s\}$.

To prove Theorem~\ref{thm:main} we use the method of interlacing polynomials, a powerful method that gained a lot of recent attention due to its role in the proof of the longstanding Kadison Singer Conjecture by Marcus, Spielman and Srivastava~\cite{KadisonSinger}. A real-rooted polynomial $f\in \mathbb{R}[t]$ with roots $s_k\leq s_{k-1}\leq \cdots \leq s_1$ is said to \textbf{interlace} a real-rooted polynomial $g\in\mathbb{R}[t]$ with roots $t_m\leq t_{m-1}\leq \cdots \leq t_1$ and we write $f\preceq g$ if
\[
 \ldots \leq s_2\leq t_2\leq s_1\leq t_1 \, .
\]
In Section~\ref{sec:preliminaries} we collect necessary preliminaries on interlacing polynomials, the operator $U_r^d$ as well as symmetric decompositions. 

We call a symmetric decomposition $(p,q)$ \textbf{interlacing} if $q\preceq p$. In Section~\ref{sec:interlacing} we investigate interlacing symmetric decompositions of $\Urd{r}{d+1}h(t)$. We provide a simple characterization (Proposition~\ref{lem:essential}) and use it to prove the following.
\begin{thm}\label{thm:main2}
Let $h(t)$ be a polynomial of degree $s\leq d$ with nonnegative coefficients and such that $h(t)=t^{s}h(1/t)$. Then $\Urd{r}{d+1}  h(t)$ has a nonnegative interlacing symmetric decomposition for all $r\geq \max\{s,d+1-s\}$.
\end{thm}
We furthermore show that $\Urd{r}{d+1}  h(t)$ has a nonnegative interlacing symmetric decomposition for all polynomials $h(t)$ with nonnegative coefficients and $\deg h(t)\leq \frac{d+1}{2}$ whenever $r\geq d+1$ (Proposition~\ref{prop:small}).

We conclude in Section~\ref{sec:Ehrharttheory} by applying our results to $h^\ast$-polynomials of dilated lattice polytopes.  In particular, we obtain that the $h^\ast$-polynomial of every dilation of a $d$-dimensional lattice polytope of degree $s$ has a real-rooted symmetric decomposition whenever the dilation factor $r$ satisfies $r\geq \max \{s,d+1-s\}$ (Corollary~\ref{cor:Ehrrealrooted}). If the polytope is Gorenstein then this decomposition is also interlacing (Corollary~\ref{cor:Ehrinterlacing}). 

\section{Preliminaries}\label{sec:preliminaries}
\subsection{Symmetric Decompositions}
For a polynomial $h(t)\in \R[t]$ of degree at most $d$, we define
\[
\mathcal{I}_d (h(t)) = t^dh(1/t) \, .
\]
If $\mathcal I _d (h(t))=h(t)$ then $h(t)$ is called \textbf{symmetric} or \textbf{palindromic} (with center of symmetry at $d/2$.) From basic linear algebra it follows that for every polynomial $h(t)$ of degree at most $d$ there exist uniquely determined symmetric polynomials $p(t) = t^d p(1/t)$ and $q(t) = t^{d-1}q(1/t)$ such that $h(t) = p(t) + tq(t)$. We call
$(p,q)$ the \textbf{symmetric decomposition} of $h(t)$. The following was observed by Br\"and\'en and Solus~\cite{branden2018symmetric}.
\begin{lem}[{\cite{branden2018symmetric}}]\label{lem:solus}
Let $h(t)$ be a polynomial of degree at most $d$. Then its symmetric decomposition $(p,q)$ is given by
\[
p(t) = \frac{h(t)-t\mathcal{I}_d (h)}{1-t} \quad \text{ and } \quad q(t) = \frac{\mathcal{I}_d(h)-h(t)}{1-t} \, .
\]
\end{lem}
In~\cite{stapledon2009inequalities} Stapledon considered also more general symmetric decompositions. For a polynomial $h(t)$ of degree $s\leq d$ the \textbf{co-degree} of $h(t)$ is defined by $\ell:=d+1-s$. 

\begin{lem}[{\cite{stapledon2009inequalities}}]\label{lem:extendedStapledon}
Let $h(t)=\sum _{i=0}^s h_it^i$ be a polynomial of degree $s\leq d$ and let $(p,q)$ be its symmetric decomposition. Then there are uniquely determined symmetric polynomials $p_\ell(t)=\sum_{i=0}^{d}p_{\ell,i}t^i$ and $q_\ell(t)=\sum_{i=0}^{s-1}q_{\ell,i}t^i$ satisfying $p_\ell(t)=t^{d}p(1/t)$ and $q_\ell(t)=t^{s-1}q_\ell(1/t)$ such that
\begin{equation}\label{eq:extendedStapledon}
(1+t+\cdots +t^{\ell-1})h(t)=p_\ell(t)+t^\ell q_\ell(t) \, ,
\end{equation}
namely
\[
p_{\ell,i} \ = \ h_0+\cdots + h_i - h_{d}-\cdots - h_{d+1-i} \, 
\]
and
\[
q_{\ell,i} \ = \ -h_0-\cdots -h_i + h_s+\cdots + h_{s-i} \, ,
\]
where $h_i := 0$ whenever $i\not \in 
\{0,1,\ldots,s\}$. In particular, $p_\ell (t)$ equals $p(t)$ in the symmetric decomposition of $h(t)$.
\end{lem}
Underlying the arguments in~\cite{stapledon2009inequalities} is the following observation.
\begin{obs}\label{obs:obs3}
The polynomial $h(t)$ as given in Lemma~\ref{lem:extendedStapledon} satisfies the Properties \hyperlink{H}{(H)} and \hyperlink{S}{(S)} if and only if $p_\ell (t)=p(t)$ and $q_\ell (t)$ have nonnegative coefficients.
\end{obs}

\subsection{Interlacing polynomials}
A real-rooted polynomial $f\in \mathbb{R}[t]$ with roots $s_k\leq s_{k-1} \leq \cdots \leq s_1$ is said to \textbf{interlace} a real-rooted polynomial $g\in\mathbb{R}[t]$ with roots $t_m\leq t_{m-1} \leq \ldots \leq t_1$, and we write $f\preceq g$, if
\[
\ldots \leq s_2 \leq t_2 \leq s_1 \leq t_1 \, .
\]
In particular, $\deg g=m$ is equal to either $\deg f=k$ or $\deg f +1 =k+1$. For technical reasons we also set $0\preceq f$ and $f\preceq 0$ for all real-rooted polynomials $f$. The following lemma collects some basic facts about interlacing polynomials that can, for example, be found in~\cite{wagner1992total}.
\begin{lem}[{\cite[Section 3]{wagner1992total}}]\label{lem:basicsinterlace}
Let $f,g,h\in \R [t]$ be real-rooted polynomials with positive leading coefficients. Then
\begin{itemize}
\item[(i)] $g\preceq f$ if and only if $cg\preceq df$ for all $c,d\neq 0$.
\item[(ii)] $h\preceq f$ and $h\preceq g$ implies $h\preceq f+g$.
\item[(iii)] $f\preceq h$ and $g\preceq h$ implies $f+g\preceq h$.
\item[(iv)] $g\preceq f$ if and only if $f\preceq tg$, if $f$ and $g$ have only nonpositive roots.
\end{itemize}
\end{lem}

A sequence of polynomials $(f_1,\ldots, f_m)$ is called an \textbf{interlacing sequence} if $f_i\preceq f_j$ whenever $i\leq j$. The following results are due to Br\"and\'en~\cite{Petter2,Brandenunimodal}.
\begin{lem}[{\cite[Lemma 8.3]{Brandenunimodal}}]\label{lem:commoninterlacer}
If $(f_1,f_2,\ldots, f_m)$ and $(g_1,g_2,\ldots, g_m)$ are two sequences of interlacing polynomials with positive leading coefficients then
\[
f_1g_m+f_2g_{m-1}+\cdots + f_mg_1
\]
is real-rooted.
\end{lem}
In particular, if $(f_1,f_2,\ldots, f_m)$ is an interlacing sequence of polynomials with positive leading coefficients then
\[
c_1f_1+c_2f_2+\cdots + c_mf_m
\]
is real-rooted for every choice of $c_1,\ldots,c_m\geq 0$. 
\begin{lem}[{\cite[Lemma 2.3]{Petter2}}]\label{lem:interlacingends}
Let $f_1,\ldots, f_m$ be polynomials such that $f_i\preceq f_{i+1}$ for all $1\leq i\leq m-1$ and such that $f_1\preceq f_m$. Then $f_i\preceq f_j$ for all $1\leq i\leq j\leq m$, that is, $(f_1,\ldots,f_m)$ is an interlacing sequence.
\end{lem}

In~\cite{branden2018symmetric} Br\"and\'en and Solus studied real-rooted and interlacing symmetric decompositions. A symmetric decomposition $(p,q)$ is called \textbf{real-rooted} if $p$ and $q$ are both real-rooted. If moreover $q\preceq p$ then the symmetric decomposition is called \textbf{interlacing}. The following theorem by Br\"and\'en and Solus~\cite{branden2018symmetric} characterizes interlacing symmetric decompositions.
\begin{thm}[{\cite[Theorem 2.7]{branden2018symmetric}}]\label{thm:solus2}
Let $h(t)$ be a polynomial of degree at most $d$ and let $(p,q)$ be its symmetric decomposition. If $p(t)$ and $q(t)$ have only nonnegative coefficients, then the following are equivalent.
\begin{itemize}
\item[(i)] $q\preceq p$.
\item[(ii)] $p\preceq h$.
\item[(iii)] $q\preceq h$.
\item[(iv)] $\mathcal I _d (h)\preceq h$.
\end{itemize}
\end{thm}

\subsection{Two operators}
A fundamental fact about generating function is that for every sequence $\{a_n\}_{n\geq 0}$ that eventually agrees with a polynomial of degree at most $d$ there is a unique polynomial $h(t)$ such that
\[
\sum _{n\geq 0}a_{n}t^n = \frac{h(t)}{(1-t)^{d+1}} \, .
\]
Vice versa, via this equation every polynomial $h(t)$ defines such a sequence (see, e.g.,~\cite[Corollary 4.3.1]{EC1}). 

If $\{a_n\}_{n\geq 0}$ that eventually agrees with a polynomial of degree at most $d$ then the same holds for $\{a_{rn}\}_{n\geq 0}$. Thus there is a unique polynomial $\Urd{r}{d+1} h (t)$ such that
\[
\sum _{n\geq 0}a_{rn}t^n = \frac{\Urd{r}{d+1}  h (t)}{(1-t)^{d+1}}.
\] 
We observe that the map $h(t)\mapsto \Urd{r}{d+1}  h (t)$ defines a linear operator $\R[t]\rightarrow \R[t]$. Moreover, $h(t)$ has degree at most $d$ if and only if $\{a_n\}_{n\geq 0}$ agrees with a polynomial of degree at most $d$ for all $n\geq 0$. In particular, in this case also $\{a_{rn}\}_{n\geq 0}$ is given by a polynomial of degree at most $d$, and therefore also $\Urd{r}{d+1}  h (t)$ has degree at most $d$. In particular, the symmetric decomposition of $\Urd{r}{d+1}  h (t)$ exists.

To every Laurent series $f(t)=\sum _{n=-\infty}^\infty a_nt^n$ and any natural number $r\geq 1$ there are uniquely defined Laurent series $f_0,f_1,\ldots,f_{r-1}$ such that
\[
f(t) \ = \ f_0(t^r)+rf_1(t^r)+\cdots + t^{r-1}f_{r-1}(t^r) \, .
\]
For all $0\leq i\leq r-1$ we define $f\hri{r}{i}=f_i$. Then $\hri{r}{i}$ defines a linear operator on Laurent series. This operator extends the definition given in~\cite{jochemko2018} from formal power series to Laurent series. (Also compare~\cite{GilRobins} where these operators have been studied in a more general setup.) 

The following result relates the two operators $\Urd{r}{d+1} $ and $\hri{r}{0}$.
\begin{lem}[{\cite{Beck2010,BrentiWelker}}]\label{lem:dilatedh}
For integers $r\geq 1$ and $d\geq 0$ and any polynomial $h(t)$
\[
\Urd{r}{d+1}  h (t) = (h(t)(1+t+\cdots +t^{r-1})^{d+1})\hri{r}{0} \, .
\]
\end{lem}
A key element in~\cite{jochemko2018} was the following set of polynomials: For $r\geq 1$, $d\geq 0$  and $0\leq i\leq r-1$ let
\[
a_d \hri{r}{i} (t)=((1+t+\cdots + t^{r-1})^d)\hri{r}{i} \, .
\]
The following result was proved in~\cite{jochemko2018} (also compare Fisk~\cite[Example 3.76]{fisk2006polynomials}).
\begin{prop}[{\cite[Proposition 3.4]{jochemko2018}}]\label{prop:ans}
For all integers $r,d\geq 1$ the polynomials
\[
\left(a_d\hri{r}{r-1}, a_d\hri{r}{r-2},\ldots, a_d \hri{r}{0}\right)
\]
form an interlacing sequence.
\end{prop}
Equivalently, by Lemma~\ref{lem:basicsinterlace},
\[
\left(a_d \hri{r}{0},ta_d\hri{r}{r-1}, ta_d\hri{r}{r-2},\ldots, ta_d \hri{r}{0}\right)
\]
is an interlacing sequence of polynomials. With this definition, Lemma~\ref{lem:dilatedh}, can be reformulated as follows.
\begin{lem}[{\cite{jochemko2018}}]\label{lem:reform}
For integers $r,d\geq 1$ and any polynomial $h(t)$
\[
\Urd{r}{d+1}  h (t)=h\hri{r}{0}a_{d+1}\hri{r}{0}+h\hri{r}{1}ta_{d+1}\hri{r}{r-1}+\cdots +h\hri{r}{r-1}ta_{d+1}\hri{r}{1} \, .
\]
\end{lem}

\section{Symmetric decomposition of $U_r^{d+1} h (t)$}\label{sec:symmetricdecomp}
In this section we provide explicit formulas for the symmetric decomposition of $U_r ^{d+1}h(t)$.

We need the following slight generalization of Lemma~\ref{lem:extendedStapledon}. 
\begin{prop}\label{prop:decomp2}
Let $h(t)=\sum _{i=0}^s h_it^i$ be a polynomial of degree $s\leq d$, let $(p,q)$ be its symmetric decomposition and let $r \geq \ell=d+1-s$ be an integer. Then there are uniquely determined symmetric polynomials $v_r(t)=\sum_{i=0}^{d}v_{r,i}t^i$ and $w_r(t)=\sum_{i=0}^{r+s-\ell -1}w_{r,i}t^i$ satisfying $v_r(t)=t^{d}v_r(1/t)$ and $w_r(t)=t^{r+s-\ell -1}w_r(1/t)$ such that
\begin{equation}\label{eq:extendedStapledon}
(1+t+\cdots +t^{r-1})h(t)=v_r(t)+t^\ell w_r(t) \, ,
\end{equation}
namely
\[
v_{r,i} \ = \ h_0+\cdots + h_i - h_{d}-\cdots - h_{d-i+1} \, 
\]
and
\[
w_{r,i} \ = \ -h_{\ell -r}-h_{\ell +1 -r}-\cdots -h_{\ell + i-r} + h_s+h_{s-1}+\cdots + h_{s-i} \, ,
\]
where $h_i := 0$ whenever $i\not \in 
\{0,1,\ldots,s\}$. In particular, $v_r(t)=p(t)$ for all $r\geq \ell$.
\end{prop}
\begin{proof}
The argument goes along the lines of the proof of Lemma~\ref{lem:extendedStapledon} given in~\cite{stapledon2009inequalities}. Let $f(t)=\sum _i f_i t^i=(1+t+\cdots + t^{r-1})h(t)$. Then $f(t)$ is a polynomial of degree $r+s-1$ with
\[
f_i = h_i+h_{i-1}+\cdots +h_{i-r+1} \, ,
\]
where $h_i := 0$ whenever $i\not \in 
\{0,1,\ldots,s\}$. Since $d+1=\ell +s$, for all $i\geq \ell$ we obtain
\begin{eqnarray*}
v_{r,i}+w_{r,i-\ell}&=&h_0+\cdots + h_i - h_{d}-\cdots - h_{d-i+1}\\&&-h_{\ell -r}-h_{\ell +1 -r}-\cdots -h_{i-r} + h_s+h_{s-1}+\cdots + h_{s-i+\ell}\\
&=&h_0+\cdots + h_i - h_{d}-\cdots - h_{s+\ell -i}\\&&-h_{\ell -r}-h_{\ell +1 -r}-\cdots -h_{i-r} + h_s+h_{s-1}+\cdots + h_{s-i+\ell}\\
 &=&h_{i-r+1}+\cdots +h_i\\
 &=&f_i \, ,
\end{eqnarray*}
and
\[
v_{r,i}=h_0+\cdots +h_i = f_i
\]
for $i<\ell$. Thus, Equation~\eqref{eq:extendedStapledon} is satisfied. Moreover, by taking $v_{r,i}$ and $w_{r,i}$ as defined above,
\begin{eqnarray*}
v_{r,i}-v_{r,d-i}&=&h_0+\cdots +h_i -h_d-\cdots -h_{d-i+1}-h_0-\cdots -h_{d-i}+h_{d}+\cdots +h_{i+1}\\
&=&0\, 
\end{eqnarray*}
and
\begin{eqnarray*}
w_{r,i}-w_{r,r+s-\ell -1 -i}&=&-h_{\ell-r}-h_{\ell -r+1}-\cdots - h_{\ell -r+i}+h_s+h_{s-1}+\cdots +h_{s-i}\\
&&+h_{\ell -r}+h_{\ell -r+1}+\cdots + h_{s-1-i}-h_s-h_{s-1}-\cdots -h_{\ell -r +1+i}\\
&=&0 \, 
\end{eqnarray*}
which shows $v_r(t)=t^{d}v_r(1/t)$ and $w_r(t)=t^{r+s-\ell -1}w_r(1/t)$, respectively. Uniqueness of $v_r(t)$ and $w_r(t)$ with the assumed properties is easily verified. \qedhere
\end{proof}

\begin{prop}\label{prop:representation}\label{prop:formula}
Let $h(t)$ be a polynomial of degree  $s\leq d$ and let $(p,q)$ be its symmetric decomposition. Let $r\geq 1$ and let $(\tilde{p},\tilde{q})$ be the symmetric decomposition of $U_r^{d+1} h (t)$. Then 
\[
\tilde{p}(t)=(p(t)(1+t+\cdots + t^{r-1})^d)\hri{r}{0}
\]
and 
\[
\tilde{q}(t)=\frac{1}{t}((h(t)(1+t+\cdots + t^{r-1})-p(t))(1+t+\cdots + t^{r-1})^d))\hri{r}{0} \, .
\]
If furthermore $r\geq \ell=d+1-s$ then
\[
\tilde{q}(t)=\frac{1}{t}(t^\ell w_r(t)(1+t+\cdots + t^{r-1})^d))\hri{r}{0} \, ,
\]
where $w_r(t)$ is defined as in Proposition~\ref{prop:decomp2}.
\end{prop}
\begin{proof}
Let $f(t)$ denote $\Urd{r}{d+1}  h(t)$. Then, by Lemma~\ref{lem:solus}, 
\[
\tilde{p}(t)=\frac{f(t)-t\mathcal{I}_d (f(t))}{1-t} \, .
\]
We have 
\begin{eqnarray*}
f(t)&=&((p(t)+tq(t))(1+t+\cdots + t^{r-1})^{d+1}))\hri{r}{0}
\end{eqnarray*}
and
\begin{eqnarray*}
t\mathcal{I}_d (f(t))&=&t^{d+1}(h(1/t)(1+1/t+\cdots + 1/t^{r-1})^{d+1}))\hri{r}{0}\\
&=&(t^{r(d+1)}(p(1/t)+1/tq(1/t))(1+1/t+\cdots + 1/t^{r-1})^{d+1}))\hri{r}{0}\\
&=&(t(p(t)+q(t))(1+t+\cdots + t^{r-1})^{d+1}))\hri{r}{0} \, .
\end{eqnarray*}
We thus obtain
\begin{eqnarray*}
f(t)-t\mathcal{I}_d (f(t))&=&((p(t)-tp(t))(1+t+\cdots + t^{r-1})^{d+1}))\hri{r}{0} \\
&=&((1-t^r)p(t)(1+t+\cdots + t^{r-1})^{d}))\hri{r}{0} \\
&=&(1-t)(p(t)(1+t+\cdots + t^{r-1})^{d}))\hri{r}{0}
\end{eqnarray*}
which proves the claimed formula for $\tilde{p}(t)$. For $\tilde{q}(t)$ we observe that
\begin{eqnarray}
t\tilde{q}(t)=f(t)-\tilde{p}(t)=((h(t)(1+t+\cdots +t^{r-1})-p(t))(1+t+\cdots + t^{r-1})^d)\hri{r}{0} \, .
\end{eqnarray}
If $r\geq \ell$ then, by Proposition~\ref{prop:decomp2}, 
\[
h(t)(1+t+\cdots +t^{r-1})-p(t)=t^\ell w_r(t)
\]
and the claim follows.
\end{proof}
\begin{cor}\label{cor:nonnegativity}
Let $h(t)$ be a polynomial of degree at most $d$ and let $(\tilde p, \tilde{q})$ be the symmetric decomposition of $\Urd{r}{d+1} h(t)$. 
\begin{itemize}
\item[(i)] If $h(t)$ satisfies Property \hyperlink{H}{(H)} then $\tilde p$ has nonnegative coefficients.
\item[(ii)] If $h(t)$ has nonnegative coefficients, satisfies Property \hyperlink{S}{(S)} and $r\geq \ell$ then $\tilde q$ has nonnegative coefficients.
\end{itemize}
\end{cor}
\begin{proof}
By Observation~\ref{obs:obs3}, $p(t)$ has nonnegative coefficients whenever $h(t)$ satisfies Property \hyperlink{H}{(H)}. From Proposition~\ref{prop:formula} we see that in this case also $\tilde p(t)$ has nonnegative coefficients. If $r\geq \ell$, then by Proposition~\ref{prop:decomp2}
\[
w_{r,i}=-h_{\ell -r}-h_{\ell -r +1}-\cdots -h_{\ell+i-r}+h_s+h_{s-1}+\cdots +h_{s-i}\geq -h_0-h_1-\cdots -h_i+h_s+h_{s-1}+h_{s-i}\geq 0
\]
where we used that $h_i\geq 0$ for all $i$ and that $h(t)$ satisfies Property \hyperlink{S}{(S)}. Thus, again by Proposition~\ref{prop:formula}, we see that $\tilde q(t)$ has nonnegative coefficients.
\end{proof}
The bound given in Corollary~\ref{cor:nonnegativity} (ii) is optimal. To see this consider a polynomial $h(t)$ of degree $s$ with $h(0)>0$ and symmetric decomposition $(p,q)$. Let $(\tilde{p},\tilde{q})$ be the symmetric decomposition of $\Urd{r}{d+1}h$. By Proposition~\ref{prop:decomp2}, 
\[
t\tilde{q}=(((1+t+\cdots + t^{r-1})h(t)-p(t))(1+t+\cdots +t^{r-1})^d)\hri{r}{0} \, .
\]
If $r<\ell$, then the leading coefficient of $(1+t+\cdots + t^{r-1})h(t)-p(t)$ equals $-p_d=-p_0=-h_0<0$ which is equal to the leading coefficient of $t\tilde{q}$.

\section{Real-rootedness}\label{sec:realrootedness}
The goal of this section is to prove Theorem~\ref{thm:main}. We use the following lemma.
\begin{lem}\label{lem:arealrooted}
Let $g(t)=\sum _{i=0}^d g_it^i$ be a polynomial of degree at most $d$ with nonnegative coefficients and let $\ell\geq 1$ be such that 
\[
g_{0}\leq g_{1}\leq \cdots \leq g_{\ell -1} \, \quad \text{ and } \quad g_{d+1-\ell}\geq \cdots \geq g_{d-1}\geq g_{d} \, .
\]
Then 
\[
\Urd{r}{d}  g (t)=\left(g(t)(1+t+\cdots + t^{r-1})^d\right)\hri{r}{0}
\]
is real-rooted for all $r\geq \max\{d+1-\ell,\frac{d+1}{2}\}$.
\end{lem}
\begin{proof}
By Theorem~\ref{lem:reform}
\begin{eqnarray}
\Urd{r}{d}  g (t)&=&g\hri{r}{0}a_{d}\hri{r}{0}+g\hri{r}{1}ta_{d}\hri{r}{r-1}+\cdots +g\hri{r}{r-1}ta_{d}\hri{r}{1}\, .
\end{eqnarray}
Since $r\geq \frac{d+1}{2}$, $g\hri{r}{i}$ has degree at most $1$ for all $i$. More precisely,
\[
g\hri{r}{i}=\begin{cases}g_i+g_{r+i}t \quad & \text{ if }  0\leq i\leq d-r\\
g_i \quad & \text{ if } d-r+1\leq i\leq r-1
\end{cases}
\]
Since $d-r\leq \ell -1$, by assumption we have $g_{0}\leq g_{1}\leq \cdots \leq g_{d-r}$ and $g_{r}\geq \cdots \geq g_{d-1}\geq g_{d}$. From that we see that
\[
(g\hri{r}{r-1}, g\hri{r}{r-2},\ldots ,g\hri{r}{0})
\]
is an interlacing sequence. Since, by Proposition~\ref{prop:ans}, $(a_{d}\hri{r}{0},ta_{d}\hri{r}{r-1},\ldots, ta_{d}\hri{r}{1})$ is an interlacing sequence, the claim follows by Lemma~\ref{lem:commoninterlacer}.
\end{proof}

\begin{proof}[Proof of Theorem~\ref{thm:main}]
Let $(p,q)$ be the symmetric decomposition of $h(t)$ and $(\tilde{p},\tilde{q})$ be the symmetric decomposition of $\Urd{r}{d+1}h(t)$. By Proposition~\ref{prop:representation}, 
\[
\tilde{p}(t)=\left(p(t)(1+t+\cdots + t^{r-1})^d\right)\hri{r}{0} \, .
\]
By Lemma~\ref{lem:extendedStapledon},
\begin{equation}\label{eq:proofmain1}
p_i=h_0+h_1+\cdots + h_i-h_d-h_{d-1}-\cdots  -h_{d-i+1}\, .
\end{equation}
Therefore, since $h(t)$ satisfies Property \hyperlink{H}{(H)}, $p(t)$ and thus also $\tilde{p}(t)$ have only nonnegative coefficients. Since $h(t)$ has nonnegative coefficients and $h_i=0$ for all $i\geq s+1$ we furthermore obtain from equation \eqref{eq:proofmain1} 
\[
p_{0}\leq p_{1}\leq \cdots \leq p_{\ell -1}\quad \text{ and } \quad p_{d+1-\ell}\geq \cdots \geq p_{d-1}\geq p_{d}\, ,
\] 
where $\ell=d+1-s$ denotes the co-degree as usual. Since $\max \{s,\frac{d+1}{2}\}=\max \{d+1-\ell,\frac{d+1}{2}\}$ we obtain that $\tilde{p}(t)=\Urd{r}{d} p(t)$ is real-rooted for all $r\geq \max \{s,\frac{d+1}{2}\}$ by Lemma~\ref{lem:arealrooted}. In particular, $\tilde{p}(t)$ is real-rooted for all $r\geq \max \{s,d+1-s\}\geq \max \{s,\frac{d+1}{2}\}$.

To see that $\tilde{q}(t)$ is real-rooted we recall that since $r\geq \ell$, by Proposition~\ref{prop:representation},
\[
t\tilde{q}(t)=\left(t^\ell w_r(t)(1+t+\cdots + t^{r-1})^d\right)\hri{r}{0} \, .
\]
Let
\[
f(t)=t^\ell w_r(t) \, .
\]
Then $f(t)=\sum _{i=0}^{r+s-1}f_i t^i$ is a polynomial of degree $r+s-1$ and has nonnegative coefficients as was argued in the proof of Corollary~\ref{cor:nonnegativity}. In particular, also $\tilde{q}(t)$ has only nonnegative coefficients. We distinguish two cases:

If $s\leq \ell$ then
\[
f\hri{r}{i}=\begin{cases}
f_{i+r}t & \text{ if } 0\leq i \leq s-1\\
f_i & \text{ if } s\leq i \leq r-1 \, .
\end{cases}
\]
In particular, $(f\hri{r}{r-1},f\hri{r}{r-2},\ldots,f\hri{r}{0})$ is an interlacing sequence. Therefore, by Lemma~\ref{lem:commoninterlacer},
\[
t\tilde{q}(t)=f\hri{r}{0}a_d\hri{r}{0}+f\hri{r}{1}ta_d \hri{r}{r-1}+\cdots +f\hri{r}{r-1}ta_d \hri{r}{1}
\]
is real-rooted.

If $\ell \leq s$ then
\begin{equation}\label{eq:case2}
f\hri{r}{i}=\begin{cases}
f_{i+r}t & \text{ if } 0\leq i \leq \ell -1\\
f_i+f_{i+r}t & \text{ if } \ell \leq i \leq s-1\\
f_i & \text{ if } s\leq i \leq r-1 \, .
\end{cases}
\end{equation}
We observe that by Proposition~\ref{prop:decomp2}
\begin{eqnarray*}
f_i&=&\begin{cases}0 & \text{ if }i\leq \ell -1\\
w_{r,i-\ell} & \text{ if } i\geq \ell
\end{cases}\\
&=&\begin{cases}0 & \text{ if }i\leq \ell -1\\ 
-h_{\ell -r}-h_{\ell +1-r}-\cdots -h_{i-r}+h_s+h_{s-1}+\cdots+h_{s-i+\ell} & \text{ if } i\geq \ell
\end{cases}
\end{eqnarray*}
In particular, since $r\geq s$ we have 
\[
f_\ell \leq f_{\ell +1} \leq \cdots \leq f_{s-1}
\]
and therefore by symmetry of $w_r (t)$ also
\[
f_{r+\ell}\geq f_{r+\ell +1}\geq \cdots \geq f_{r+s-1} \, .
\]
Together with equation~\eqref{eq:case2} this implies that $(f\hri{r}{r-1},f\hri{r}{r-2},\ldots,f\hri{r}{0})$ is an interlacing sequence. Therefore, again as in the first case, it follows by Lemma~\ref{lem:commoninterlacer} that $t \tilde{q}$ and thus $\tilde{q}$ is real-rooted. This finishes the argument.
\end{proof}
\begin{rem}
The proof of Theorem~\ref{thm:main} actually shows something slightly stronger if $h(t)$ has small degree: If  $\deg h\leq \frac{d+1}{2}$ then  $\tilde{p}(t)$ in the symmetric decomposition of $\Urd{r}{d+1}h(t)$ is already real-rooted whenever $r\geq \frac{d+1}{2}$. 
\end{rem}

\section{Interlacing Symmetric Decompositions}\label{sec:interlacing}
In this section we investigate when $\Urd{r}{d+1} h(t)$ has an interlacing symmetric decomposition. We obtain the following characterization.
\begin{prop}\label{lem:essential}
Let $h(t)$ be a polynomial of degree $s\leq d$ with nonnegative coefficients that satisfies Property \hyperlink{H}{(H)} and \hyperlink{S}{(S)}. For all $r\geq \ell=d+1-s$, $\Urd{r}{d+1} h(t)$ has an interlacing symmetric decomposition if and only if 
\[
\Urd{r}{d+1} h(t)\preceq \Urd{r}{d+1} (\mathcal I _{d+1} h(t)) \, .
\]
\end{prop}
\begin{proof}
Since $r\geq \ell$, by Theorem~\ref{thm:solus2} and Corollary~\ref{cor:nonnegativity}, $\Urd{r}{d+1} h(t)$ has an interlacing symmetric decomposition if and only if $\mathcal{I}_d (\Urd{r}{d+1} h(t))$ interlaces $\Urd{r}{d+1} h(t)$. Since all coefficients of $\Urd{r}{d+1} h(t)$ are nonnegative this is the case if and only if $\Urd{r}{d+1} h(t)$ interlaces $t\mathcal{I}_d (\Urd{r}{d+1} h(t))=\mathcal{I}_{d+1} (\Urd{r}{d+1} h(t))$ which equals
\begin{eqnarray*}
\mathcal I _{d+1} (\Urd{r}{d+1} h(t))&=&t^{d+1}(h(1/t)(1+1/t+\cdots +1/t^{r-1})^{d+1})\hri{r}{0}\\
&=&(t^{r(d+1)}h(1/t)(1+1/t+\cdots +1/t^{r-1})^{d+1})\hri{r}{0}\\
&=&(t^{(d+1)}h(1/t)(1+t+\cdots +t^{r-1})^{d+1})\hri{r}{0}\\
&=&\Urd{r}{d+1}  (\mathcal I_{d+1} h(t))\, . 
\end{eqnarray*}\qedhere
\end{proof}
This characterization is used to prove the following result and Theorem~\ref{thm:main2}. 
\begin{prop}\label{prop:small}
Let $h(t)$ be a polynomial of degree $s\leq \frac{d+1}{2}$ with nonnegative coefficients that satisfies Property \hyperlink{S}{(S)}. Then for all $r\geq d+1$, $\Urd{r}{d+1} h(t)$ has an interlacing symmetric decomposition. 
\end{prop}
\begin{proof}
By Lemma~\ref{lem:reform}, we have
\[
\Urd{r}{d+1} h(t)=h_0a_{d+1}\hri{r}{0}+h_1ta_{d+1}\hri{r}{r-1}+\ldots + h_sta_{d+1}\hri{r}{r-s}
\]
and 
\[
\Urd{r}{d+1} (\mathcal I _{d+1} h(t))=h_0ta_{d+1}\hri{r}{r-(d+1)}+h_1ta_{d+1}\hri{r}{r-d}+\ldots + h_sta_{d+1}\hri{r}{r-(d+1-s)} \, .
\]
Since $s\leq d+1-s$, the polynomials $(a_{d+1}\hri{r}{0},ta_{d+1}\hri{r}{r-1},\ldots, ta_{d+1}\hri{r}{r-s},ta_{d+1}\hri{r}{r-(d+1-s)},\ldots,ta_{d+1}\hri{r}{r-(d+1)})$ form an interlacing sequence and thus we obtain $\Urd{r}{d+1} h(t)\preceq \Urd{r}{d+1} (\mathcal I _{d+1} h(t))$ by repeated application of Lemma~\ref{lem:basicsinterlace}. Furthermore, since Property \hyperlink{H}{(H)} is automatically satisfied whenever $\deg h \leq \frac{d+1}{2}$, $\Urd{r}{d+1} h(t)$ has an interlacing symmetric decomposition by Lemma~\ref{lem:essential}.
\end{proof}
To prove Theorem~\ref{thm:main2} we use the following generalization of Lemma~\ref{prop:ans}. This should be compared with Fisk~\cite[Proposition 3.72 \& Example 3.76]{fisk2006polynomials}. Compare also with Br\"and\'en~\cite[Section 8]{Brandenunimodal}.
\begin{lem}\label{lem:ans2}
Let $h(t)$ be a polynomial and for all $0\leq i<r$ and $d\geq 0$ let
\[
a_{h,d}\hri{r}{i}:=(h(t)(1+t+\cdots +t^{r-1})^d)\hri{r}{i} \, .
\]
Then 
\[
a_{h,d+1}\hri{r}{i}=a_{h,d}\hri{r}{0}+a_{h,d}\hri{r}{1}+\cdots +a_{h,d}\hri{r}{i}+ta_{h,d}\hri{r}{i+1}+\cdots +ta_{h,d}\hri{r}{r-1} \, .
\]
In particular, if $h(t)$ is of degree $s$ and has only nonnegative coefficients then
\[
(a_{h,d}\hri{r}{r-1}, a_{h,d}\hri{r}{r-2},\ldots, a_{h,d} \hri{r}{0})
\]
is an interlacing sequence of polynomials for all $r\geq s$.
\end{lem}
\begin{proof}
The proof combines the ideas of the proofs of \cite[Lemma 3.2]{jochemko2018} and \cite[Proposition 3.4]{jochemko2018}. 

Let $f(t)=(1+t+\cdots +t^{r-1})$ and $g(t)=h(t)(1+t+\cdots + t^{r-1})^{d}$. We compute
\begin{eqnarray*}
h(t)(1+t+\cdots +t^{r-1})^{d+1}&=&f(t)g(t)\\
&=&(f\hri{r}{0}(t^r)+tf\hri{r}{1}(t^r)+\cdots +t^{r-1}f\hri{r}{r-1}(t^r))\\
&&\cdot(g\hri{r}{0}(t^r)+tg\hri{r}{1}(t^r)+\cdots +t^{r-1}g\hri{r}{r-1}(t^r))\\
&=&\sum _{i=0}^{r-1}\left(t^i\sum _{k+\ell=i}f\hri{r}{k}(t^r)g\hri{r}{\ell}(t^r)+t^{i+r}\sum _{k+\ell=i+r}f\hri{r}{k}(t^r)g\hri{r}{\ell}(t^r)\right)\\
&=&\sum _{i=0}^{r-1}t^i\left(\sum _{k+\ell=i}a_{h,d}\hri{r}{\ell}(t^r)+\sum _{k+\ell=i+r}(ta_{h,d}\hri{r}{\ell})(t^r)\right) \, .
\end{eqnarray*}
In particular, 
\begin{equation}\label{eq:recursion}
a_{h,d+1}\hri{r}{i}=a_{h,d}\hri{r}{0}+a_{h,d}\hri{r}{1}+\cdots +a_{h,d}\hri{r}{i}+ta_{h,d}\hri{r}{i+1}+\cdots +ta_{h,d}\hri{r}{r-1} \, .
\end{equation}

Now we assume that $h(t)$ has only nonnegative coefficients and degree $s$. To prove the second part of the claim we use induction on $d$. If $d=0$ then 
\[
a_{h,0}\hri{r}{i}=\begin{cases} h_i & \text{ if } 1\leq i \leq r-1\\
h_0 & \text{ if } i=0  \text{ and } r>s\\
h_0+th_s & \text{ if } i=0 \text{ and }r=s \, .
\end{cases}
\]
where $h_i:=0$ if $i\not \in \{0,1,\ldots,s\}$. Thus, in all cases, $(a_{h,d}\hri{r}{r-1}, a_{h,d}\hri{r}{r-2},\ldots, a_{h,d} \hri{r}{0})$ is trivially an interlacing sequence.

If we now assume by induction that $(a_{h,d}\hri{r}{r-1}, a_{h,d}\hri{r}{r-2},\ldots, a_{h,d} \hri{r}{0})$ is an interlacing sequence then, by applying a standard interlacing preserving map (see, e.g., \cite[Proposition 2.2]{jochemko2018} or \cite[Theorem 2.3]{savage2015s}) together with Equation~\eqref{eq:recursion}, also $(a_{h,d+1}\hri{r}{r-1}, a_{h,d+1}\hri{r}{r-2},\ldots, a_{h,d+1} \hri{r}{0})$ is an interlacing sequence.
\end{proof}

\begin{proof}[Proof of Theorem~\ref{thm:main2}]
Since $r\geq s$, by Lemma~\ref{lem:ans2}, $(a_{h,d}\hri{r}{0},ta_{h,d}\hri{r}{r-1},\ldots,ta_{h,d}\hri{r}{1},ta_{h,d}\hri{r}{0})$ is an interlacing sequence. We furthermore observe 
\[
\Urd{r}{d+1}(t^kh(t))=(t^kh(t)(1+t+\cdots +t^{r-1})^{d+1})\hri{r}{0}=ta_{h,d+1}\hri{r}{r-k}
\]
for $1\leq k\leq r$. In particular, since $d+1-s\leq r$,
\[
\Urd{r}{d+1}h(t)=a_{h,d}\hri{r}{0} \preceq ta_{h,d}\hri{r}{r-d+s-1}=\Urd{r}{d+1}(t^{d-s+1}h(t))=\Urd{r}{d+1}(\mathcal I _{d+1} h(t)) \, .
\]
We moreover observe that Properties \hyperlink{H}{(H)} and \hyperlink{S}{(S)} are automatically satisfied for symmetric polynomials with nonnegative coefficients. Thus, by Proposition~\ref{lem:essential}, $\Urd{r}{d+1}h(t)$ has an interlacing decomposition for $r\geq \max\{s,d+1-s\}$.
\end{proof}
The bound in Theorem~\ref{thm:main2} is optimal. Indeed, if we consider polynomials of the form $h(t)=1+t^s$, and let $(\tilde{p},\tilde{q})$ be the symmetric decomposition of $\Urd{r}{d+1}h(t)$ then, by Corollary~\ref{cor:nonnegativity}, $\tilde{p}$ has only nonnegative coefficients for all $r\geq 0$. If $\tilde{q}$ interlaces $\tilde{p}$ then also the coefficients of $\tilde{q}$ are nonnegative as all roots have to be nonpositive. At the end of Section~\ref{sec:symmetricdecomp} we argued that $r$ thus needs to be greater or equal to $d+1-s$. Furthermore, if $\tilde{q}$ interlaces $\tilde{p}$, then $\Urd{r}{d+1}h(t)$ is real-rooted. In~\cite[Section 5]{jochemko2018} it was shown that for $h(t)=1+t^s$ the bound $r\geq s$ is sharp for $\Urd{r}{d+1}h(t)$ to be real-rooted.

We have seen that $\Urd{r}{d+1} h(t)$ has an interlacing decomposition for all $r\geq d+1$ if $h(t)$ has small degree (Proposition~\ref{prop:small}) and for $r\geq \max\{s,d+1-s\}$ if $h(t)$ is symmetric (Theorem~\ref{thm:main2}). Based on computational experiments we believe that these bounds might constitute general uniform bounds.
\begin{quest}
Let $h(t)$ be a polynomial of degree $s\leq d$ with nonnegative coefficients that satisfies Properties \hyperlink{H}{(H)} and \hyperlink{S}{(S)}. Does $\Urd{r}{d+1}h(t)$ have an interlacing symmetric decomposition for all $r\geq d+1$ (possibly even for $r\geq \max\{s,d+1-s\}$)?
\end{quest}

\section{Ehrhart theory}\label{sec:Ehrharttheory}
A lattice polytope in $\mathbb{R}^d$ is defined as the convex hull of finitely many points in the integer lattice $\mathbb{Z}^d$. Ehrhart theory is concerned with counting lattice points in lattice polytopes and their integer dilates. For a comprehensive introduction to Ehrhart theory we recommend~\cite{BR}. 

In~\cite{ehrhartRational} Ehrhart proved that the number of lattice points $\vert nP\cap \mathbb{Z}^d\vert$ in the $n$-th dilate of a lattice polytope $P\subset \R^d$ agrees with a polynomial $\Ehr_P (n)$ of degree $\dim P$ for all integers $n\geq 0$. The polynomial $\Ehr_P(n)$ is called the \textbf{Ehrhart polynomial} of $P$. 

The $h^\ast$-polynomial $h^\ast _P(t)$ of a polytope $P$ encodes $\Ehr _P(n)$ in a particular basis. If $P$ is $d$-dimensional, then the relation between $h^\ast _P(t)$ and $\Ehr _P(n)$ is given by
\[
\sum _{n\geq 0} E_P (n)t^n \ = \ \frac{h^\ast _P (t)}{(1-t)^{d+1}} \, .
\]
In particular, the degree of $h^\ast _P$ is at most $d$. 

The coefficients of the Ehrhart polynomial can be negative and rational in general. In contrast, Stanley~\cite{Stanley78} showed that the $h^\ast$-polynomial has only nonnegative integer coefficients. Subsequently, Hibi~\cite{HibiH} and Stanley~\cite{StanleyS} proved two fundamental families of inequalities for the coefficients of the $h^\ast$-polynomial, namely Property \hyperlink{H}{(H)} and Property \hyperlink{S}{(S)}. In summary, the $h^\ast$-polynomial satisfies all conditions of Theorem~\ref{thm:main}. Furthermore, for any $d$-dimensional lattice polytope $P$ and any dilation factor $r\in \mathbb{Z}_{>0}$
\[
\frac{h^\ast _{rP} (t)}{(1-t)^{d+1}} \ = \ \sum _{n\geq 0} E_{rP} (n)t^n \ = \ \sum _{n\geq 0} E_{P} (nr)t^n \ = \ \frac{\Urd{r}{d+1}h^\ast _{P} (t)}{(1-t)^{d+1}} \, ,
\]
that is, $\Urd{r}{d+1}h^\ast _{P} (t)$ is equal to the $h^\ast$-polynomial of the $r$-dilate of the lattice polytope $P$. The following corollary is therefore an immediate consequence of Theorem~\ref{thm:main}.
\begin{cor}\label{cor:Ehrrealrooted}
Let $P$ be a $d$-dimensional lattice polytope and let $h^\ast _P (t)=h_0+h_1t+\cdots +h_st^s$, $\deg h^\ast _P=s\leq d$, be the $h^\ast$-polynomial of $P$. Then $h^\ast _{rP}(t)$ has a real-rooted symmetric decomposition for all integers $r\geq \max\{s,d+1-s\}$.
\end{cor}
A lattice polytope $P$ in $\R^d$ is called \textbf{reflexive} if
\[
P \ = \ \{\mathbf{x}\in \R^d\colon \mathbf{Ax}\leq \mathbf{1}\} \, ,
\]
up to a translation by a vector in $\mathbb{Z}^d$, where $\mathbf{A}$ is an integer matrix and $\mathbf{1}$ denotes the all $1$s vector. A polytope $P$ is called \textbf{Gorenstein} if $P$ has an integer dilate $\ell P$ that is reflexive. Hibi~\cite{Hibireflexive} showed that a polytope $P$ is reflexive if and only if $h_P^\ast(t)=t^dh^\ast _P (1/t)$.  Stanley~\cite{StanleyGorenstein} extended this characterization to Gorenstein polytopes.

As an immediate consequence of Theorem~\ref{thm:main2} we obtain the following.
\begin{cor}\label{cor:Ehrinterlacing}
Let $h^\ast _P(t)$ be the $h^\ast$-polynomial of a $d$-dimensional Gorenstein lattice polytope. Then $h^\ast _{rP} (t)$ has an interlacing decomposition whenever $r\geq \max\{s,d+1-s\}$.
\end{cor}

\textbf{Acknowledgements:} The author wants to thank Petter Br\"and\'en and Liam Solus for interesting discussions, and Florian Kohl as well as the anonymous referee for careful reading and many helpful comments on the manuscript. This article was finished while the author was a participant at the semester program ``Algebraic and Enumerative Combinatorics'' at Mittag-Leffler-Institute in Djursholm in spring 2020. This work was partially supported by the Wallenberg AI, Autonomous Systems and Software Program (WASP) funded by the Knut and Alice Wallenberg Foundation, as well as by grant 2018-03968 of the Swedish Research Council.

\bibliographystyle{siam}
\bibliography{SymmetricDecomp}

\begin{thebibliography}{10}

\bibitem{athanasiadis2020}
{\sc C.~A. Athanasiadis}, {\em Face numbers of uniform triangulations of
  simplicial complexes}, arXiv preprint arxiv.org/abs/2003.13372,  (2020).

\bibitem{jochemkozonotopes}
{\sc M.~Beck, K.~Jochemko, and E.~McCullough}, {\em {$h^\ast$}-polynomials of
  zonotopes}, Trans. Amer. Math. Soc., 371 (2019), pp.~2021--2042.

\bibitem{BR}
{\sc M.~Beck and S.~Robins}, {\em Computing the {C}ontinuous {D}iscretely:
  {I}nteger-point {E}numeration in {P}olyhedra}, Undergraduate Texts in
  Mathematics, Springer, New York, 2007.

\bibitem{Beck2010}
{\sc M.~Beck and A.~Stapledon}, {\em On the log-concavity of {H}ilbert series
  of {V}eronese subrings and {E}hrhart series}, Math. Z., 264 (2010),
  pp.~195--207.

\bibitem{BetkeMcMullen}
{\sc U.~Betke and P.~McMullen}, {\em Lattice points in lattice polytopes},
  Monatsh. Math., 99 (1985), pp.~253--265.

\bibitem{Petter2}
{\sc P.~Br\"{a}nd\'{e}n}, {\em On linear transformations preserving the
  {P}\'{o}lya frequency property}, Trans. Amer. Math. Soc., 358 (2006),
  pp.~3697--3716.

\bibitem{Brandenunimodal}
\leavevmode\vrule height 2pt depth -1.6pt width 23pt, {\em Unimodality,
  log-concavity, real-rootedness and beyond}, in Handbook of enumerative
  combinatorics, Discrete Math. Appl. (Boca Raton), CRC Press, Boca Raton, FL,
  2015, pp.~437--483.

\bibitem{branden2018symmetric}
{\sc P.~Br{\"a}nd{\'e}n and L.~Solus}, {\em Symmetric decompositions and
  real-rootedness}, arXiv preprint arXiv:1808.04141,  (2018).

\bibitem{Braununimodal}
{\sc B.~Braun}, {\em Unimodality problems in {E}hrhart theory}, in Recent
  trends in combinatorics, vol.~159 of IMA Vol. Math. Appl., Springer, [Cham],
  2016, pp.~687--711.

\bibitem{Brentiunimodal}
{\sc F.~Brenti}, {\em Unimodal, log-concave and {P}\'{o}lya frequency sequences
  in combinatorics}, Mem. Amer. Math. Soc., 81 (1989), pp.~viii+106.

\bibitem{BrentiWelker}
{\sc F.~Brenti and V.~Welker}, {\em The {V}eronese construction for formal
  power series and graded algebras}, Adv. Appl. Math., 42 (2009), pp.~545--556.

\bibitem{Bruns}
{\sc W.~Bruns, J.~Gubeladze, and N.~V. Trung}, {\em Normal polytopes,
  triangulations, and {K}oszul algebras}, J. Reine Angew. Math., 485 (1997),
  pp.~123--160.

\bibitem{Diaconis2}
{\sc P.~Diaconis and J.~Fulman}, {\em Carries, shuffling, and symmetric
  functions}, Adv. Appl. Math., 43 (2009), pp.~176--196.

\bibitem{ehrhartRational}
{\sc E.~Ehrhart}, {\em Sur les poly\`edres rationnels homoth\'etiques \`a
  {$n$}\ dimensions}, C. R. Acad. Sci. Paris, 254 (1962), pp.~616--618.

\bibitem{ferroni2020ehrhart}
{\sc L.~Ferroni}, {\em On the {E}hrhart polynomial of minimal matroids}, arXiv
  preprint arXiv:2003.02679,  (2020).

\bibitem{fisk2006polynomials}
{\sc S.~Fisk}, {\em Polynomials, roots, and interlacing}, arXiv preprint
  math/0612833,  (2006).

\bibitem{GilRobins}
{\sc J.~B. Gil and S.~Robins}, {\em Hecke operators on rational functions.
  {I}}, Forum Math., 17 (2005), pp.~519--554.

\bibitem{haglund2019real}
{\sc J.~Haglund and P.~B. Zhang}, {\em Real-rootedness of variations of
  {E}ulerian polynomials}, Adv. Appl. Math., 109 (2019), pp.~38--54.

\bibitem{Hering}
{\sc M.~S. Hering}, {\em Syzygies of toric varieties}, ProQuest LLC, Ann Arbor,
  MI, 2006.
\newblock Thesis (Ph.D.)--University of Michigan.

\bibitem{HibiH}
{\sc T.~Hibi}, {\em Some results on {E}hrhart polynomials of convex polytopes},
  Discrete Math., 83 (1990), pp.~119--121.

\bibitem{Hibireflexive}
\leavevmode\vrule height 2pt depth -1.6pt width 23pt, {\em Dual polytopes of
  rational convex polytopes}, Combinatorica, 12 (1992), pp.~237--240.

\bibitem{higashitani2009}
{\sc A.~Higashitani}, {\em Unimodality of {$\delta$}-vectors of lattice
  polytopes and two related properties}, Eur. J. Math., 5 (2019), pp.~333--355.

\bibitem{higashitani19}
{\sc A.~Higashitani, K.~Jochemko, and M.~Michalek}, {\em Arithmetic aspects of
  symmetric edge polytopes}, Mathematika, 65 (2019), pp.~763--784.

\bibitem{hlavacek2020subdivisions}
{\sc M.~Hlavacek and L.~Solus}, {\em Subdivisions of shellable complexes},
  arXiv preprint arXiv:2003.07328,  (2020).

\bibitem{jochemko2018}
{\sc K.~Jochemko}, {\em On the real-rootedness of the {V}eronese construction
  for rational formal power series}, Int. Math. Res. Not. IMRN,  (2018),
  pp.~4780--4798.

\bibitem{KadisonSinger}
{\sc A.~W. Marcus, D.~A. Spielman, and N.~Srivastava}, {\em Interlacing
  families {II}: {M}ixed characteristic polynomials and the {K}adison-{S}inger
  problem}, Ann. of Math. (2), 182 (2015), pp.~327--350.

\bibitem{savage2015s}
{\sc C.~Savage and M.~Visontai}, {\em The 𝐬-{E}ulerian polynomials have only
  real roots}, Trans. Amer. Math. Soc., 367 (2015), pp.~1441--1466.

\bibitem{Schepers}
{\sc J.~Schepers and L.~Van~Langenhoven}, {\em Unimodality questions for
  integrally closed lattice polytopes}, Ann. Comb., 17 (2013), pp.~571--589.

\bibitem{SolusSimplices}
{\sc L.~Solus}, {\em Simplices for numeral systems}, Trans. Amer. Math. Soc.,
  371 (2019), pp.~2089--2107.

\bibitem{StanleyGorenstein}
{\sc R.~P. Stanley}, {\em Hilbert functions of graded algebras}, Adv. Math., 28
  (1978), pp.~57--83.

\bibitem{Stanley78}
\leavevmode\vrule height 2pt depth -1.6pt width 23pt, {\em Decompositions of
  rational convex polytopes}, Ann. Discrete Math., 6 (1980), pp.~333--342.
\newblock Combinatorial mathematics, optimal designs and their applications
  (Proc. Sympos. Combin. Math. and Optimal Design, Colorado State Univ., Fort
  Collins, Colo., 1978).

\bibitem{Stanleyunimodal}
\leavevmode\vrule height 2pt depth -1.6pt width 23pt, {\em Log-concave and
  unimodal sequences in algebra, combinatorics, and geometry}, in Graph theory
  and its applications: {E}ast and {W}est ({J}inan, 1986), vol.~576 of Ann. New
  York Acad. Sci., New York Acad. Sci., New York, 1989, pp.~500--535.

\bibitem{StanleyS}
\leavevmode\vrule height 2pt depth -1.6pt width 23pt, {\em On the {H}ilbert
  function of a graded {C}ohen-{M}acaulay domain}, J. Pure Appl. Algebra, 73
  (1991), pp.~307--314.

\bibitem{EC1}
\leavevmode\vrule height 2pt depth -1.6pt width 23pt, {\em Enumerative
  combinatorics. {V}olume 1}, vol.~49 of Cambridge Studies in Advanced
  Mathematics, Cambridge University Press, Cambridge, second~ed., 2012.

\bibitem{stapledon2009inequalities}
{\sc A.~Stapledon}, {\em Inequalities and {E}hrhart $\delta$-vectors}, Trans.
  Amer. Math. Soc., 361 (2009), pp.~5615--5626.

\bibitem{wagner1992total}
{\sc D.~G. Wagner}, {\em Total positivity of {H}adamard products}, J. Math.
  Anal. Appl., 163 (1992), pp.~459--483.

\bibitem{Zhang}
{\sc P.~B. Zhang}, {\em On the real-rootedness of the local {$h$}-polynomials
  of edgewise subdivisions}, Electron. J. Combin., 26 (2019), pp.~Paper 1.52,
  6.

\end{thebibliography}
\end{document}